\documentclass[11pt,letterpaper,reqno]{amsart}

\usepackage[T1]{fontenc}
\usepackage[utf8]{inputenc}
\usepackage{lmodern}
\usepackage{amsmath,amssymb,amsthm,mathtools,mathrsfs}
\usepackage{microtype}
\usepackage{mathtools}
\mathtoolsset{showonlyrefs}
\usepackage{esint}
\usepackage[numbers,sort&compress]{natbib}
\usepackage[colorlinks, citecolor=blue, linkcolor=red]{hyperref}

\allowdisplaybreaks[2]
\numberwithin{equation}{section}
\newtheorem{theorem}{Theorem}[section]
\newtheorem{proposition}[theorem]{Proposition}
\newtheorem{lemma}[theorem]{Lemma}
\newtheorem{corollary}[theorem]{Corollary}
\theoremstyle{remark}
\newtheorem{remark}[theorem]{Remark}

\newcommand{\R}{\mathbb R}
\newcommand{\Id}{\mathrm{Id}}
\newcommand{\diver}{\operatorname{div}}
\newcommand{\dist}{\operatorname{dist}}

\newcommand{\Hess}{\nabla^2}

\newcommand{\dd}{\,dx}
\newcommand{\ds}{\,d\mathcal H^{n-1}}

\title[The Gagliardo--Nirenberg equation]
{The Gagliardo--Nirenberg equation on Riemannian manifolds}

\author[G. Catino]{Giovanni Catino}
\address{G. Catino (corresponding author), Dipartimento di Matematica,
Politecnico di Milano, Piazza Leonardo da Vinci 32, 20133 Milano, Italy}
\email{giovanni.catino@polimi.it}
\author[D. D. Monticelli]{Dario D. Monticelli}
\address{D. D. Monticelli, Dipartimento di Matematica,
Politecnico di Milano, Piazza Leonardo da Vinci 32, 20133 Milano, Italy}
\email{dario.monticelli@polimi.it}
\author[A. Roncoroni]{Alberto Roncoroni}
\address{A. Roncoroni, Dipartimento di Matematica,
Politecnico di Milano, Piazza Leonardo da Vinci 32, 20133 Milano, Italy}
\email{alberto.roncoroni@polimi.it}

\date{}
\hypersetup{
  pdftitle={Sharp annular rigidity for the Gagliardo--Nirenberg equation},
  pdfauthor={Giovanni Catino, Dario D. Monticelli, Alberto Roncoroni},
  pdfsubject={Classification at the linear annular-mass threshold},
  pdfkeywords={Gagliardo--Nirenberg equation, annular rigidity, reflection symmetry}
}

\begin{document}
\begin{abstract}
We establish geometric and analytic rigidity results for the
Gagliardo--Nirenberg equation. On a complete noncompact Riemannian
manifold with nonnegative Ricci curvature, the existence of a positive
critical point of the Gagliardo--Nirenberg quotient forces the manifold to be Euclidean and the solution to be extremal. No minimizing assumption or volume growth hypothesis is required, thereby strengthening earlier rigidity results. In Euclidean space, we prove that all positive solutions of the normalized equation whose nonlinear
mass on expanding annuli has sublinear growth are extremals. The linear threshold is optimal, as shown by explicit profiles that are constant
in one direction. We also prove a sharp pointwise estimate involving
the solution and its gradient, valid without assumptions at infinity,
with equality at one point characterizing the extremals.
The proofs use weighted divergence identities and a reflection
argument controlled by annular mass. Explicit constructions reveal the geometric and variational limits of rigidity: nonflat quotient critical points exist beyond the curvature hypothesis, while the normalized equation admits positive
solutions even under strictly positive sectional curvature and maximal volume growth.
\end{abstract}
\maketitle


\section{Introduction}\label{sec:introduction}

The classification of positive solutions to nonlinear elliptic
equations is closely connected with the equality cases of sharp
functional inequalities. An extremal satisfies an Euler--Lagrange
equation, but the converse requires a rigidity theorem: a critical
point need not minimize the functional. On a Riemannian manifold,
there is a further question of whether the existence of such a
solution determines the ambient geometry. In this paper, we address
these questions for the Gagliardo--Nirenberg equation, proving
rigidity for critical points under nonnegative Ricci
curvature, a sharp annular classification in Euclidean space, and a
pointwise estimate with rigidity in the equality case.

For equations with critical Sobolev growth, the model semilinear equation is 
\begin{equation}\label{eq_crit}
-\Delta u= u^{\frac{n+2}{n-2}} \, ,  \quad u>0 \qquad \text{in } \mathbb{R}^n\, ,     
\end{equation}
for $n\geq 3$. The classification of solutions of \eqref{eq_crit} has been proved in \cite{Obata,GNN} under an a priori decay assumption on $u$ and in \cite{CaffarelliGidasSpruck} (see also \cite{Chen-Li}) without additional assumptions via the method of moving planes. In the quasilinear setting, the equation is
\begin{equation}\label{eq_crtip}
-\Delta_p u =u^{p^\ast -1} \,,  \quad u>0 \qquad  \text{in } \mathbb{R}^n\, ,
\end{equation}
where 
$$
1<p<n\quad \text{and} \quad p^{\ast}=\frac{np}{n-p}\, .    
$$
The classification of finite energy solutions of \eqref{eq_crtip} was completed in \cite{Sciunzi,Vetois} (see also \cite{CFR}). Beyond the energy framework, very recently, the classification of solutions of \eqref{eq_crtip} has been obtained in \cite{Zhang} (see also \cite{CatinoMonticelliRoncoroni,CG,Ou,SunWang,Vetois_bis} for previous and partial results).  

On complete Riemannian manifolds with nonnegative Ricci curvature, these
questions combine analytic classification with geometric rigidity.
For the critical semilinear equation, classification and rigidity results were obtained in \cite{CatinoMonticelli,CFP,FogagnoloMalchiodiMazzieri,SunWang} under suitable assumptions (growth or finite energy) or when the Riemannian manifold has dimension at most five (see also \cite{Ou,Vetois_bis}). In each case, a nontrivial solution forces the manifold to be Euclidean. Differential identities involving a suitable
power of the solution, its Hessian, and the Ricci tensor make it
possible to obtain this conclusion without relying on Euclidean
reflection symmetries. Related rigidity phenomena have also been investigated beyond the Riemannian critical Sobolev setting. For singular or degenerate critical equations arising from Caffarelli–Kohn–Nirenberg inequalities, see e.g. \cite{CC,CMR,CP,DEL}. In CR and sub-Riemannian geometries, Liouville-type rigidity and classification results were obtained in the Heisenberg group in \cite{CatinoLiMonticelliRoncoroni,FVet,JerisonLee,Liu,Prajapat}, and on Sasakian manifolds in \cite{CatinoMonticelliRoncoroniWang}.
These results illustrate how identities adapted to the equation can turn
information about a single solution into rigidity of the underlying
space.

\

Our starting point is the family of \emph{sharp Gagliardo--Nirenberg
inequalities} proved in \cite{DelPinoDolbeault} (see also \cite{CorderoNazaretVillani, DELM}). Given $n\geq 2$, if  $1<p<\frac{n}{n-2}$\footnote{the upper endpoint is understood as $+\infty$ when $n=2$.} then the following inequality holds
$$
\left(\int_{\mathbb{R}^n}|f|^{2p}\, dx\right)^{\frac{1}{2p}}\leq C_{n,p} \left(\int_{\mathbb{R}^n}|\nabla f|^2\, dx \right)^{\frac{\tau}{2}} \left(\int_{\mathbb{R}^n}|f|^{p+1}\, dx\right)^{\frac{1-\tau}{p+1}}\, ,
$$
for all $f\in\mathcal D^p(\mathbb{R}^n)$ where 
$$
 \mathcal D^p(\mathbb{R}^n)
 :=\bigl\{f\in L^{p+1}(\mathbb{R}^n)\cap L^{2p}(\mathbb{R}^n)\, :\,
                |\nabla f|\in L^2(\mathbb{R}^n)\bigr\},
$$
and 
\begin{equation}\label{eq:critical-GN-parameters}
 \tau:=\frac{n(p-1)}{p[n+2-p(n-2)]}\in(0,1).
\end{equation}
Moreover, the positive extremals are given by 
$$
\left(\frac{1}{c_1+c_2|x-x_0|^2}\right)^{\frac{1}{p-1}}\,, 
$$
for some point $x_0\in\R^n$, and constants $c_1,c_2>0$.

On a complete Riemannian manifold $(M^n,g)$, define
\begin{equation}\label{eq:critical-GN-energy-space}
 \mathcal D^p(M)
 :=\bigl\{f\in L^{p+1}(M)\cap L^{2p}(M)\, :\,
                |\nabla f|\in L^2(M)\bigr\},
\end{equation}
with weak gradient and Lebesgue spaces taken with respect to the canonical volume form $d\mu_g$. The corresponding Gagliardo--Nirenberg quotient is
\begin{equation}\label{eq:critical-GN-quotient}
 \mathscr G_{n,p}[f]
 =\frac{\|\nabla f\|_{L^2(M)}^\tau
        \|f\|_{L^{p+1}(M)}^{1-\tau}}
       {\|f\|_{L^{2p}(M)}},
\end{equation}
defined for nonzero $f\in\mathcal D^p(M)$ with
$\|\nabla f\|_{L^2(M)}>0$. We call a smooth function $u$ in
this domain a critical point of $\mathscr G_{n,p}$ if
\[
 \left.\frac{d}{dt}\right|_{t=0}\mathscr G_{n,p}[u+t\varphi]=0
 \qquad\text{for every }\varphi\in C_c^\infty(M).
\]
Our first theorem classifies all such positive critical points and the
manifolds on which they can exist.

\begin{theorem}
\label{thm:riem-gn-critical-rigidity}
Let $(M^n,g)$, $n\ge2$, be a complete, noncompact Riemannian manifold with $\operatorname{Ric}_g\ge0$, and let $p$ satisfy
\begin{equation}\label{eq:range}
 1<p<\frac{n}{n-2}.
\end{equation}
If $u\in\mathcal D^p(M)$ is a positive critical
point of $\mathscr G_{n,p}$, then $(M,g)$ is isometric to $\R^n$ with its standard metric and 
\begin{equation}\label{eq:critical-GN-classification}
 u(x)=\left(\frac{1}{c_1+c_2|x-x_0|^2}\right)^{\frac{1}{p-1}}\,, 
\end{equation}
for some point $x_0\in\R^n$, and constants $c_1,c_2>0$.
\end{theorem}

Thus every positive critical point is an extremal, and its existence
already determines both the metric and the function. This extends
the Euclidean extremal classification in \cite[Theorem 1]{DelPinoDolbeault} to a
geometric rigidity statement for arbitrary positive critical points.
It also strengthens, for the family considered here, the rigidity
associated with sharp Gagliardo--Nirenberg inequalities on Riemannian manifolds.
The connection between Euclidean best constants and geometric
rigidity goes back to Ledoux \cite{Ledoux} for Sobolev inequalities
and to Xia \cite{Xia} for Gagliardo--Nirenberg inequalities.
Moreover, in \cite[Theorem 3.1]{BaloghKristaly} the authors determined sharp Gagliardo--Nirenberg inequalities and geometric rigidity on complete, noncompact Riemannian manifolds satisfying
$$
\mathrm{Ric}_g \geq 0 \quad \text{and} \quad \mathrm{AVR}_g:=\lim_{r\to+\infty}\frac{\mathrm{Vol}_g(B_r(o))}{\omega_n r^n}>0\, ,
$$
where $o\in M$ is some fixed point and $\omega_n$ denotes the volume of the unit ball in $\mathbb{R}^n$. 

The proof of Theorem \ref{thm:riem-gn-critical-rigidity} uses the integral relations supplied by quotient
criticality. After normalization, these relations give a weighted
mean constraint for the auxiliary function
$$
v:=u^{1-p}\, . 
$$
Combined with a divergence identity (see \eqref{eq:riem-gn-P-divergence} below), this constraint forces the traceless Hessian of $v$ and the Ricci contribution to vanish from which the rigidity follows. 

The curvature condition is essential. In
Proposition~\ref{ex:explicit-GN-all-dimensions} below we construct, for every $n\geq 2$ and $1<p<\frac{n}{n-2}$, a complete nonflat warped product carrying an explicit positive critical point in $\mathcal D^p(M)$.
These manifolds have two ends, negative Ricci curvature in the radial
direction, and polynomial volume growth of order greater than their
dimension. The examples show that quotient criticality alone does
not impose Euclidean geometry when the curvature condition is
removed.

\

Our second theorem concerns positive entire solutions in Euclidean
space without assuming membership in the energy space $\mathcal{D}^p(\mathbb{R}^n)$. We consider
the \emph{normalized Gagliardo-Nirenberg equation}
\begin{equation}\label{eq:main}
 -\Delta u=u^{2p-1}-u^p\,, \quad u>0 \qquad\text{in }\R^n\, ,
\end{equation}
where $n\geq 2$ and $1<p<\frac{n}{n-2}$. The Euler--Lagrange equation to \eqref{eq:critical-GN-quotient} with $M=\mathbb{R}^n$ can
be brought to this form by rescaling. The normalized profiles are
\begin{equation}\label{eq:profile}
 \mathcal U_{x_0}(x)=
 \left(\frac{D}{2p}+\frac{(p-1)^2}{2D}|x-x_0|^2\right)^{-\frac{1}{p-1}}\, , 
 \qquad x_0\in\R^n,
\end{equation}
where 
\begin{equation}\label{eq:constants}
 D:=2p-n(p-1)=n-p(n-2)>0\,  ,
\end{equation}
and only translations remain free in this normalization. The equation \eqref{eq:main} also admits solutions of a different kind, including constants, periodic solutions, and profiles depending on fewer than
$n$ variables. A classification therefore requires a condition that
excludes these solutions. We impose an integral condition on annuli. For $R>0$, we define the \emph{annular mass} of $u$, i.e. 
\begin{equation}\label{eq:mass}
 M(R):=\int_{A_R}u^{p+1}\dd \, ,
\end{equation}
where 
$$
A_R:=B_{2R}(0)\setminus B_R(0) \, . 
$$
In the next theorem we prove rigidity under sublinear growth of $M(R)$.

\begin{theorem}\label{thm:main}
Let $u\in C^2(\R^n)$ be a solution of \eqref{eq:main}
with $1<p<\frac{n}{n-2}$. If
\begin{equation}\label{eq:main-hyp}
 \liminf_{R\to\infty}\frac{M(R)}{R}=0,
\end{equation}
then $u=\mathcal U_{x_0}$ for some $x_0\in\R^n$.

Conversely, every profile \eqref{eq:profile} solves \eqref{eq:main}
and satisfies $M(R)=O(R^{-s})$ as $R\to\infty$, where 
\begin{equation}\label{eq:s}
 s=\frac{2(p+1)}{p-1}-n=\frac{D+2}{p-1}>0.
\end{equation}
\end{theorem}

Condition~\eqref{eq:main-hyp} only asks that the normalized annular
mass tend to zero along a sequence of radii. In particular,
$M(R)=O(R^\theta)$ for any $\theta<1$ is sufficient. No minimizing,
stability, finite-energy, or pointwise decay assumption is imposed.
In fact, finite energy and decay follow from the annular condition,
as made explicit in the following consequence.

\begin{corollary}\label{cor:finite-mass}
Let $u\in C^2(\R^n)$ be a solution of \eqref{eq:main} with $1<p<\frac{n}{n-2}$. The following conditions are equivalent:
\begin{enumerate}
\item[\textup{(i)}] $\displaystyle
 \liminf_{R\to\infty}\frac{M(R)}{R}=0$;
\item[\textup{(ii)}] $u\in L^{p+1}(\R^n)$;
\item[\textup{(iii)}] $u\in\mathcal D^p(\R^n)$;
\item[\textup{(iv)}] $u(x)\to0$ as $|x|\to\infty$;
\item[\textup{(v)}] $u=\mathcal U_{x_0}$ for some $x_0\in\R^n$.
\end{enumerate}
Moreover, a nonnegative classical solution satisfying any of
\textup{(i)}--\textup{(iv)} is either identically zero or a profile
\eqref{eq:profile}.
\end{corollary}
The threshold in Theorem~\ref{thm:main} is optimal. A normalized
profile in $m<n$ variables solves \eqref{eq:main} and has annular mass
asymptotic to $R^{n-m}$; see
Proposition~\ref{prop:cylinders} below for further details. In particular, taking $m=n-1$ gives a solution
with exactly linear annular mass that is not a profile in $n$
variables. Consequently, condition~\eqref{eq:main-hyp}
cannot be replaced by $M(R)=O(R)$. Constant and positive periodic
solutions, described in Remark~\ref{ex:periodic} below, give further
examples with annular mass comparable to $R^n$.

The Euclidean assumption in Theorem~\ref{thm:main} is unavoidable. In 
Proposition~\ref{prop:positive-avr-counterexamples} below we construct
complete metrics with strictly positive sectional curvature
and any prescribed asymptotic volume ratio in $(0,1)$, supporting
positive radial solutions of the normalized equation with optimal decay at infinity. This construction is not in contrast with Theorem~\ref{thm:riem-gn-critical-rigidity}: solving the normalized
equation does not imply criticality of the Gagliardo-Nirenberg quotient. 

The proof of Theorem~\ref{thm:main} derives symmetry and decay from
the annular condition through a reflection argument. Ordering for
every hyperplane gives radial symmetry and decay. The uniqueness
results in \cite{PucciSerrinUniqueness,SerrinTang} (see also
\cite[proof of Theorem~4]{DelPinoDolbeault}) then identify the
profile. In contrast with the classical moving planes method starting from decay at infinity, decay is obtained here as part of the conclusion.

\

Finally, we also prove the following Euclidean sharp estimate for
every positive entire solution.

\begin{theorem}\label{thm:sharp}
Every solution $u\in C^2(\mathbb{R}^n)$ of \eqref{eq:main} with $1<p<\frac{n}{n-2}$ satisfies
\begin{equation}\label{eq:sharp-u}
 \frac{D}{2p}u^{p-1}
 +\frac{D}{2}\frac{|\nabla u|^2}{u^{p+1}}\le1
 \quad\text{in }\R^n.
\end{equation}
Furthermore, if the equality holds at one point then $u=\mathcal U_{x_0}$ for some
$x_0\in\R^n$.
\end{theorem}

The left-hand side of \eqref{eq:sharp-u} is the $P$-function introduced in the proof of Theorem \ref{thm:riem-gn-critical-rigidity} (see \eqref{eq:critical-GN-P-expressions} below). 
Every profile $\mathcal U_{x_0}$ satisfies $P\equiv1$, so the estimate and
its equality statement are sharp. Its proof starts from the same
weighted divergence identity used in the proof of Theorem \ref{thm:riem-gn-critical-rigidity}, now
with vanishing Ricci tensor, and applies a maximum principle using
some universal estimates. Moreover, the lower-dimensional profiles satisfy $P<1$ everywhere,
showing that the estimate also applies to solutions outside the
classified family.

\subsection*{Organization of the paper} The paper is organized as follows.
Section~\ref{subsec:riem-gn-critical-rigidity} proves
Theorem~\ref{thm:riem-gn-critical-rigidity}.
Section~\ref{sec:estimates} establishes the universal estimates,
and Section~\ref{sec:reflection} contains the reflection argument
and the proof of Theorem~\ref{thm:main}.
Section~\ref{sec:sharp} proves Theorem~\ref{thm:sharp}.
The final section contains the nonflat quotient critical points
in Section~\ref{sec:nonflat-GN-critical-examples}, the proof of
Corollary~\ref{cor:finite-mass} and the sharp Euclidean examples in
Section~\ref{sec:consequences}, and the positively curved examples
for the normalized equation in Section~\ref{sec:riemannian}.

\medskip

The constants denoted by $C$ may change from line to line and depend
only on $n,p$ unless stated otherwise; constants $C_u$ may also
depend on the solution.

\section[Proof of the critical-point rigidity theorem]
{Proof of Theorem~\ref{thm:riem-gn-critical-rigidity}}
\label{subsec:riem-gn-critical-rigidity}

We use $\Delta_g=\operatorname{div}_g\nabla$ and take all
differential operators and tensor norms with respect to $g$.

\begin{proof}[Proof of Theorem \ref{thm:riem-gn-critical-rigidity}]
\noindent\emph{Step 1.} Let $u\in\mathcal{D}^p(M)$ be a positive critical point of the Gagliardo--Nirenberg quotient \eqref{eq:critical-GN-quotient}, i.e. 
\begin{equation}\label{critical_point}
 \left.\frac{d}{dt}\right|_{t=0}\mathscr G_{n,p}[u+t\varphi]=0\, ,
 \quad\text{for every }\varphi\in C_c^\infty(M)\, .     
\end{equation}
For simplicity, we set
$$
 I_1:=\int_M|\nabla u|^2\,d\mu_g,\qquad
 I_2:=\int_Mu^{p+1}\,d\mu_g,\qquad
 I_3:=\int_Mu^{2p}\,d\mu_g\, ;
$$
these quantities are finite and positive since
$u\in\mathcal D^p(M)$. A direct computation using \eqref{critical_point} gives, for every $\varphi\in C_c^\infty(M)$, 
$$
 \frac{\tau}{I_1}\int_M\langle\nabla u,\nabla\varphi\rangle\,d\mu_g
 +\frac{1-\tau}{I_2}\int_Mu^p\varphi\,d\mu_g
 -\frac{1}{I_3}\int_Mu^{2p-1}\varphi\,d\mu_g=0 \, ,
$$
where $\tau$ is given by \eqref{eq:critical-GN-parameters}. Integrating by parts we obtain 
\begin{equation}\label{eq:riem-gn-EL}
 -\Delta_g u=\mu u^{2p-1}-\lambda u^p \quad \text{ in } M\, , 
 \end{equation}
 where 
 \begin{equation}\label{l_and_mu}
      \lambda:=\frac{(1-\tau)I_1}{\tau I_2}>0
 \quad \text{ and }\quad 
 \mu:=\frac{I_1}{\tau I_3}>0\, . 
 \end{equation}
We now consider the following rescaling of $u$ and $g$
\begin{equation}\label{eq:critical-GN-normalization}
 \widetilde u=\left(\frac\mu\lambda\right)^{\frac{1}{p-1}}u \quad \text{and} \quad 
\widetilde g=\frac{\lambda^2}{\mu}g\, , 
\end{equation}
then \eqref{eq:riem-gn-EL} becomes the following
\begin{equation}\label{eq:critical-GN-normalized-equation}
 -\Delta_{\widetilde g}\widetilde u
 =\widetilde u^{2p-1}-\widetilde u^p \quad \text{ in } M\, . 
\end{equation}
We observe that this rescaling of the metric preserves completeness, nonnegative Ricci curvature, and finiteness of the integrals. The Gagliardo-Nirenberg quotient is invariant under
this rescaling since by \eqref{eq:critical-GN-parameters} one has
$$
 \frac{\tau(n-2)}4+\frac{n(1-\tau)}{2(p+1)}-\frac n{4p}=0\, . 
$$
From now on, we drop the tildes and redefine $I_1,I_2,I_3$ using the rescaled solution and metric. Stationarity is preserved, and the rescaled function is nonconstant
because $I_1>0$. Comparing its Euler--Lagrange equation \eqref{eq:riem-gn-EL} with
\eqref{eq:critical-GN-normalized-equation} we have
\begin{equation}\label{eq:riem-gn-ABC}
 \frac{I_1}{I_2}=\frac{\tau}{1-\tau}
          =\frac{n(p-1)}{D(p+1)},\qquad
 \frac{I_1}{I_3}=\tau,\qquad I_3=I_1+I_2\, ,
\end{equation}
where $D$ is given by \eqref{eq:constants}.

\medskip
\noindent\emph{Step 2.} We fix some $o\in M$ and write $r(x)=d_g(o,x)$; from the Laplacian comparison (see e.g. \cite[Chapter~7]{Petersen}) we have 
$$
\Delta_g r\le \frac{n-1}{r} \quad \text{ in } M\, ,
$$ 
in the weak sense. Choose $c>0$ with
\[
 c\le\min_{\partial B_1(o)}u,\qquad
 c^{p-1}\le\frac{2D}{(p-1)^2},
\]
and set 
$$
\phi=c r^{-\frac{2}{p-1}} \quad \text{ in } M\setminus B_1(o)\, . 
$$
A direct computation, using the Laplacian comparison, shows that 
$$
\Delta_g\phi\ge\frac{2D}{(p-1)^2}c\,r^{-\frac{2p}{p-1}}
 \ge\phi^p \quad \text{ in } M\setminus B_1(o)\, , 
$$
in the weak sense. On the other hand, 
$$
\Delta_g u\le u^p \quad \text{in }  M 
$$
moreover, from the definition of $c$ and since $u>0$ we have 
$$
\phi-u\le 0\quad \text{ on } \partial B_1(o)\quad \text{ and } \quad \limsup_{r\to\infty}(\phi-u)\le0\, . 
$$ 
The weak maximum principle in $M\setminus B_1(o)$ gives 
$$
u\geq \phi \quad \text{ in } M\setminus B_1(o)\, , 
$$
i.e. 
$$
 u\ge c\,r^{-\frac{2}{p-1}} \quad \text{ in } M\setminus B_1(o) 
 $$
 i.e. 
 \begin{equation}\label{eq:riem-gn-v-growth}
 v:=u^{1-p}\le C(1+r^2)\quad\text{in }M\, ,
\end{equation}
for some constant $C>0$ depending also on $\min_{\partial B_1(o)}u$.

\medskip
\noindent\emph{Step 3.} Differentiating $v=u^{1-p}$ and using \eqref{eq:critical-GN-normalized-equation} we have 
\begin{equation}\label{eq:critical-GN-pressure-equation}
 v\Delta_g v=(p-1)(1-v)+\frac p{p-1}|\nabla v|^2.
\end{equation}
We define the $P$-function 
\begin{equation}\label{eq:critical-GN-P-definition}
P=\frac D{2p}\left(1+\frac{\Delta_g v}{p-1}\right)\, , 
\end{equation}
and the symmetric $2$-tensor
\begin{equation}\label{eq:critical-GN-tensor-definition}
 T_g:=\nabla^2v-\frac{(p-1)^2}{D}P g
     =\nabla^2v-\frac{p-1}{2p}(\Delta_gv+p-1)g\, . 
\end{equation}
By using \eqref{eq:critical-GN-pressure-equation} we have 
\begin{equation}\label{eq:critical-GN-P-expressions}
 P=\frac D{2pv}+\frac D{2(p-1)^2v}|\nabla v|^2
  =\frac D{2p}u^{p-1}
    +\frac D2\frac{|\nabla u|^2}{u^{p+1}}>0\, . 
\end{equation}
Since $u\in\mathcal{D}^p(M)$ we have $u^{p+1}P\in L^1(M)$, and, by \eqref{eq:riem-gn-ABC},
\begin{equation}\label{eq:riem-gn-P-mean}
 \int_Mu^{p+1}P\,d\mu_g
 =\frac D{2p}I_3+\frac D2 I_1
 =\frac D{2p}I_2+\frac{D(p+1)}{2p}I_1=I_2\, , 
\end{equation}
i.e. 
\begin{equation}\label{eq:critical-GN-zero-mean}
 u^{p+1}(P-1)\in L^1(M) \quad \text{ and} \quad 
 \int_Mu^{p+1}(P-1)\,d\mu_g=0.
\end{equation}
Taking the gradient of \eqref{eq:critical-GN-pressure-equation} gives
$$
 v\nabla\Delta_gv
 =\frac{2p}{p-1}\nabla^2v(\nabla v)
  -(\Delta_gv+p-1)\nabla v
 =\frac{2p}{p-1}T_g(\nabla v)\, , 
$$
where we used \eqref{eq:critical-GN-tensor-definition}. Thus,
\begin{equation}\label{eq:critical-GN-gradient-P}
 \nabla P=\frac D{(p-1)^2}v^{-1}T_g(\nabla v)\quad \text{and } \quad
 \operatorname{tr}_g T_g=(p-1)(P-1).
\end{equation}
The standard commutation formula
\cite[Chapter~9]{Petersen}, with
$(\operatorname{div}_g T_g)_j=\nabla^i(T_g)_{ij}$, gives
$$
 \operatorname{div}_g\nabla^2v
 =\nabla(\Delta_gv)+\operatorname{Ric}_g(\nabla v)\, . 
$$
It yields
$$
 \operatorname{div}_g T_g
 =\frac{p+1}{(p-1)}v^{-1}T_g(\nabla v)
   +\operatorname{Ric}_g(\nabla v)\,  
$$
and therefore,
\begin{equation}\label{eq:critical-GN-weighted-tensor}
 \operatorname{div}_g(u^{p+1}T_g)
 =u^{p+1}\operatorname{Ric}_g(\nabla v)\, . 
\end{equation}
Since 
$$ 
\langle T_g,\nabla^2v\rangle
 =|T_g|^2+\frac{(p-1)^2}{D}P\operatorname{tr}_g T_g
 =|T_g|^2+\frac{(p-1)^3}{D}P(P-1)
$$
using \eqref{eq:critical-GN-gradient-P} and \eqref{eq:critical-GN-weighted-tensor}, we have the following
\begin{equation}\label{eq:riem-gn-P-divergence}
 \operatorname{div}_g(u^2\nabla P)
 =u^{p+1}\left[
  \frac D{(p-1)^2}
   \bigl(|T_g|^2+\operatorname{Ric}_g(\nabla v,\nabla v)\bigr)
  +(p-1)P(P-1)\right]\, .
\end{equation}

\medskip
\noindent\emph{Step 4.} We fix $R\geq 1$ and we choose standard Lipschitz cutoffs $0\le\eta_R\le1$ such that 
$$
\eta_R\equiv1 \quad\text{in } B_R(o) \, , \qquad \eta_R\equiv0 \quad \text{in } M\setminus B_{2R}(o) 
$$
and 
$$
|\nabla\eta_R|\le \frac{C}{R}\, . 
$$
Testing \eqref{eq:riem-gn-P-divergence} with $\eta_R^2$ gives
\begin{align*}
 &\int_M\eta_R^2u^{p+1}\left[
  \frac D{(p-1)^2}\bigl(|T_g|^2+
     \operatorname{Ric}_g(\nabla v,\nabla v)\bigr)
  +(p-1)P(P-1)\right]d\mu_g\\
 &\qquad=-\frac{2D}{(p-1)^2}\int_M
    \eta_Ru^{p+1}\langle T_g(\nabla v),\nabla\eta_R\rangle\,d\mu_g.
\end{align*}
By Young's inequality and using $P(P-1)=(P-1)^2+(P-1)$, we obtain
\begin{align}
 &\frac D{2(p-1)^2}\int_M\eta_R^2u^{p+1}|T_g|^2\,d\mu_g
  +\frac D{(p-1)^2}\int_M\eta_R^2u^{p+1}
       \operatorname{Ric}_g(\nabla v,\nabla v)\,d\mu_g\\
 &\qquad+(p-1)\int_M\eta_R^2u^{p+1}(P-1)^2\,d\mu_g\\\notag
 &\le(p-1)\left|\int_M\eta_R^2u^{p+1}(P-1)\,d\mu_g\right|
   \\
   &\qquad+\frac{C}{R^2}
       \int_{B_{2R}(o)\setminus B_R(o)}u^{p+1}|\nabla v|^2\,d\mu_g\, ,
       \label{eq:critical-GN-cutoff-estimate}
\end{align}
for some $C>0$. The first term on the right-hand side tends to zero by dominated convergence and
\eqref{eq:critical-GN-zero-mean}.
For the second term, we have 
\begin{align*}
     \frac1{R^2}\int_{B_{2R}(o)\setminus B_R(o)}
       u^{p+1}|\nabla v|^2\,d\mu_g&=\frac{(p-1)^2}{R^2}\int_{B_{2R}(o)\setminus B_R(o)} v|\nabla u|^2\, d\mu_g \\
 &\le C\int_{M\setminus B_R(o)}|\nabla u|^2\,d\mu_g
 \longrightarrow0\, , 
\end{align*}
where we used \eqref{eq:riem-gn-v-growth}. Since every term on the left-hand side of \eqref{eq:critical-GN-cutoff-estimate} is nonnegative, taking the limit we have 
$$
T_g\equiv0\, , \quad P\equiv1 \, , \quad \mathrm{Ric}(\nabla v,\nabla v)\equiv 0 \, .
$$
In particular, from \eqref{eq:critical-GN-tensor-definition} we have
\begin{equation}\label{eq:critical-GN-hessian}
 \nabla^2v=\frac{(p-1)^2}{D}g\, . 
\end{equation}
First we recall that all the computations above were done with respect to $\widetilde u$ and $\widetilde g$. By using Tashiro's theorem \cite{Tashiro} we have that $(M,\widetilde g)$ is isometric to $\mathbb{R}^n$ with its standard metric. Moreover, from \eqref{eq:critical-GN-hessian}, we obtain
\begin{equation}\label{eq:critical-GN-radial-v}
 v(x)=v(x_0)+\frac{(p-1)^2}{2D}|x-x_0|_{\widetilde g}^2
\end{equation}
for some point $x_0$. From $P\equiv 1$ and \eqref{eq:critical-GN-P-expressions} we deduce, 
\[
v(x_0)=\frac D{2p}.
\]
Since $v={\widetilde u}^{1-p}$, the rescaled function is
\[
 \widetilde u(x)
 =\left[\frac D{2p}+\frac{(p-1)^2}{2D}|x-x_0|_{\widetilde g}^2\right]^{-\frac{1}{p-1}}.
\]
Using \eqref{eq:critical-GN-normalization} we obtain
\[
 u(x)=\left[
       \frac{\mu D}{2p\lambda}
       +\frac{\lambda(p-1)^2}{2D}|x-x_0|^2
       \right]^{-\frac{1}{p-1}}.
\]
This proves \eqref{eq:critical-GN-classification}.
\end{proof}

\section{Universal estimates in $\R^n$}\label{sec:estimates}

We first establish bounds that hold for every positive entire
solution of \eqref{eq:main}. Local elliptic regularity and positivity imply that such a solution is smooth.
In what follows, we will use the doubling lemma in \cite{PolacikQuittnerSouplet}, classical interior and boundary estimates in \cite[Chapters~6, 8, and~9]{GilbargTrudinger} and a result by Keller--Osserman in  \cite{Keller,Osserman}. In particular, we will use the following \emph{local Keller--Osserman estimate:} if 
$$
\Delta w\geq \frac{1}{2}w^p\, , \quad w>0 \quad \text{in } B_R(x) \, ,\quad \text{ with } p>1  
$$
then 
$$
 w(x)\leq c R^{-\frac{2}{p-1}}\, .
$$
The previous local estimate implies the famous \emph{Keller--Osserman Liouville theorem:} there is no positive solution of $\Delta w\geq \frac{1}{2}w^p$ in $\mathbb{R}^n$ with $p>1$. 

Let $n\geq 2$ and $1<p<\frac{n}{n-2}$, and let $u$ be a positive solution (smooth) of 
\begin{equation}\label{eq:main_bis}
    -\Delta u=u^{2p-1}-u^p \quad \text{ in } \R^n\, .
\end{equation}
We consider the superlevel set of $u$ 
$$
K:=\left\{x\in\R^n\, : \, u(x)\ge2^{-\frac{1}{p-1}}\right\}
$$
and we denote by 
$$
d_K(x):=\dist(x,K)\, .
$$

\begin{lemma}\label{lem_universal} We have, 
\begin{itemize}
    \item[$i)$] $u$ is bounded, i.e. 
    \begin{equation}\label{u_bdd}
 u\le C
 \quad\text{in }\R^n\, ;
 \end{equation}
    \item[$ii)$] $K\neq\emptyset$ and
\begin{equation}\label{eq:distance-lower}
 u(x)\ge C\left(1+d_K(x)\right)^{-\frac{2}{p-1}}\, ; 
\end{equation}
\item[$iii)$] the following gradient estimate
\begin{equation}\label{gradu_est}
    |\nabla u|\le Cu^{\frac{p+1}{2}}
 \quad\text{in }\R^n\, ,
\end{equation}
holds.
\end{itemize}
Here the constants $C>0$ depend only on $n$ and $p$.
\end{lemma}
\begin{proof} \emph{$i)$} Assume that estimate \eqref{u_bdd} fails. Then there are solutions $u_j$
and points $y_j\in\R^n$ such that $u_j(y_j)\to\infty$.
Apply the doubling lemma
\cite[Lemma~5.1 and Remark 5.2 (a)]{PolacikQuittnerSouplet} to $u_j^{p-1}$ on $\R^n$,
with doubling parameter $j$. It gives points $x_j\in\R^n$ such that
$$
A_j:=u_j(x_j)\ge u_j(y_j)
$$
and 
$$
 u_j(x)\le 2^{\frac{1}{p-1}}u_j(x_j)
 \quad\text{for all } |x-x_j|\le jA_j^{1-p}\, .
$$
Now, we rescale $u_j$ by setting
\[
 U_j(y):=A_j^{-1}u_j\left(x_j+A_j^{1-p}y\right),\qquad |y|\leq j\, .
\]
The functions $U_j$ solve
$$
 -\Delta U_j(y)=U_j^{2p-1}(y)-A_j^{1-p}U_j^p(y) \, ,\quad  |y|\leq j\, . 
$$
Moreover, 
$$
U_j(0)=A_j^{-1}u_j\left(x_j\right)=1
$$
and 
$$
0<U_j(y)\le2^{\frac{1}{p-1}} \, , \quad  |y|\leq j\, . 
$$
By using elliptic estimates and standard embeddings, we deduce that some sub-sequence of $U_j$ converges in $C^2_{\mathrm{loc}}$ to a (classical) solution $U\geq 0$ of 
$$
-\Delta U=U^{2p-1} \quad \text{ in } \R^n\, , 
$$
moreover, $U(0)=1$. This contradicts, for $n\geq 3$, the classical Gidas-Spruck Liouville theorem~\cite{GidasSpruck}, since $p<\frac{n}{n-2}$ gives 
$$
2p-1< \frac{n+2}{n-2}\, .
$$
If $n=2$, on the other hand, this contradicts the standard Liouville theorem in $\R^2$. Thus \eqref{u_bdd} holds. Interior estimates on unit balls, applied to \eqref{eq:main_bis} then yield 
\begin{equation}\label{gradu_bdd}
|\nabla u|\le C(n,p)
\quad\text{in }\mathbb R^n.
\end{equation}
\emph{$ii)$} We first prove that $K\neq\emptyset$. Suppose, by contradiction, that
$$
u(x)<2^{-\frac{1}{p-1}}
\qquad\text{for every }x\in\R^n.
$$
Then
$$
1-u^{p-1}\ge \frac12
$$
and hence
$$
\Delta u
=u^p\bigl(1-u^{p-1}\bigr)
\ge \frac12 u^p
\qquad\text{in }\R^n.
$$
By the Keller--Osserman Liouville theorem recalled before we have $u\equiv0$, which contradicts $u>0$. Thus $K\neq\emptyset$.

The equation can be written as
$\Delta u=u^{p-1}(1-u^{p-1})u$, with uniformly bounded coefficient.
Harnack's inequality on $B_2(y)$
\cite[Chapter~8]{GilbargTrudinger} therefore gives
\begin{equation}\label{harn}
\inf_{B_1(y)}u\geq c \sup_{B_1(y)}u\geq 2^{-\frac{1}{p-1}}c=:c_1(n,p)\quad\forall y\in K.
\end{equation}
Choose a positive constant $c_2=c_2(n,p)\leq c_1$ so that
\[
 c_2^{p-1}\leq\frac2{p-1}\left(\frac2{p-1}+2-n\right).
\]
The right-hand side is positive by \eqref{eq:range}.  For
$|x-y|\geq1$, the function
\[
 \phi_y(x)=c_2|x-y|^{-\frac2{p-1}}
\]
satisfies
\[
 \Delta\phi_y
 =c_2\frac2{p-1}\left(\frac2{p-1}+2-n\right)
 |x-y|^{-\frac{2p}{p-1}}\geq\phi_y^p.
\]
Since 
$$
\Delta u\leq u^p,\quad \phi_y\leq u\quad\text{on } \partial B_1(y),\quad
\limsup_{|x|\to\infty}(\phi_y-u)\leq0,
$$
comparison gives
$$
\phi_y\leq u\quad\text{on }B_1(y)^c.
$$
Fix $x\in\R^n$.  Since $K$ is nonempty and closed, there exists
$y\in K$ such that $|x-y|=d_K(x)$. If $d_K(x)\leq1$, then $x\in\overline{B_1(y)}$, and \eqref{harn} gives
\[
 u(x)\geq c_1\geq c_2\geq c_2(1+d_K(x))^{-\frac2{p-1}}.
\]
If $d_K(x)>1$, we instead apply the exterior comparison with the barrier
centered at this same point $y$:
\[
 u(x)\geq\phi_y(x)=c_2\,d_K(x)^{-\frac2{p-1}}
 \geq c_2(1+d_K(x))^{-\frac2{p-1}}.
\]
Thus, for every $x\in\R^n$, \eqref{eq:distance-lower} holds.

\medskip

\emph{$iii)$} If $d=d_K(x)\le1$, the estimate follows from
\eqref{gradu_bdd} and \eqref{eq:distance-lower}.
If $d>1$, since $B_d(x)\cap K=\emptyset$, then, arguing as in $ii)$, we have
$$
\Delta u\geq \frac12 u^p\quad\text{in }B_d(x).
$$
Applying the local Keller--Osserman estimate on balls of radius $d/4$ centered in
$B_{d/2}(x)$ gives
\begin{equation}\label{supu}
 \sup_{B_{d/2}(x)}u\le Cd^{-2/(p-1)}.
\end{equation}
Thus, the rescaled function $W(\xi)=d^{2/(p-1)}u(x+d\xi)$ on $B_{1/2}$
satisfies
\begin{equation}\label{eq:distance-rescaling}
 0<W\le C,\qquad \Delta W=W^p-d^{-2}W^{2p-1}\quad\text{on }B_{1/2}
\end{equation}
The right-hand side is uniformly bounded. The interior gradient
estimate gives $|\nabla W(0)|\le C$, whence
\[
 |\nabla u(x)|\le Cd^{-(p+1)/(p-1)}
 \le C u(x)^{(p+1)/2},
\]
by \eqref{eq:distance-lower}. This proves \eqref{gradu_est}.
\end{proof}

\begin{corollary}\label{cor:pressure-bounds}
For $v=u^{1-p}$, we have
\begin{equation}\label{eq:pressure-bounds}
 v^{-1}\le C,\qquad |\nabla v|^2\le Cv.
\end{equation}
In particular,
\begin{equation}\label{eq:pressure-growth}
 v(x)\le C_u(1+|x|^2),\qquad
 |\nabla v(x)|\le C_u(1+|x|).
\end{equation}
\end{corollary}
\begin{proof}
Inequalities \eqref{eq:pressure-bounds} follow from Lemma \ref{lem_universal}. Moreover, since $|\nabla\sqrt v|\le C$, integration along a segment from the origin proves \eqref{eq:pressure-growth}.
\end{proof}

\section{Proof of Theorem~\ref{thm:main}}
\label{sec:reflection}

We combine the classical reflection and maximum-principle arguments of
Gidas--Ni--Nirenberg \cite{GNN}
with the Pohozaev--Pucci--Serrin variational identity
\cite{PucciSerrinVariational}. The annular hypothesis supplies an
ordering between the solution and its reflection with respect to every hyperplane. The usual consequences of the strong maximum principle and Hopf's
lemma then give radial symmetry and monotonicity. Decay follows from
the same annular hypothesis, and radial uniqueness identifies the profile.

\begin{proof}[Proof of Theorem~\ref{thm:main}]
Write 
\begin{equation}\label{def_f}
f(s)=s^{2p-1}-s^p\, ,     
\end{equation} 
so that $u$ solves 
$$
-\Delta u=f(u) \quad \text{ in } \mathbb{R}^n\, . 
$$
By Lemma~\ref{lem_universal}, there exists a constant $C_0>0$ such that $0<u\leq C_0$.
Fix a sequence $R_j\to\infty$ such that $$M(R_j)/R_j\to0.$$ Fix $e\in\mathbb{S}^{n-1}$ and $t\in\R$, and consider the half-space such that $e$ is its inward unit normal, i.e. 
$$
 H=\{x\in\mathbb{R}^n\, :\, x\cdot e>t\} 
$$
and its boundary 
$$
\Pi=\{x\in\mathbb{R}^n\, :\, x\cdot e=t\}\, .
$$
We set $u^\sigma=u\circ\sigma$, where $\sigma:\mathbb{R}^n\to\mathbb{R}^n$ is the reflection with respect to $\Pi$, explicitly $\sigma(x)=x-2(x\cdot e-t)e$.

\medskip
\noindent\emph{Step 1.} We construct a function $z\in C^2_{\mathrm{loc}}(\overline H)$ such that
 \begin{equation}\label{eq:lower-solution}
 \begin{cases}
     -\Delta z=f(z) \, & \text{ in } H \\
     0<z\le\min\{u,u^\sigma\}\, & \text{ in }H\\
     z=u\, &  \text{ on }\Pi\, . 
 \end{cases}
\end{equation}
Set $h:=\min\{u,u^\sigma\}$. We have that $h\in C(\overline H)$ is a distributional
supersolution, while $0$ is a subsolution because $f(0)=0$.

For $k\ge1$, let $\Omega_k=B_k(te)\cap H$ and choose
$\chi_k\in C_c^\infty(B_{3k/4}(te))$ such that
$0\le\chi_k\le1$ and $\chi_k=1$ on $B_{k/2}(te)$.
Prescribe the continuous boundary datum
\[
 g_k=(\chi_k u)|_{\partial\Omega_k}.
\]
On the flat part of $\partial\Omega_k$ we have
$u=u^\sigma=h$, whereas on the spherical part $g_k=0$.
Thus $0\le g_k\le h$ on the whole boundary, including the
edge where its two parts meet. Moreover, every boundary point of
$\Omega_k$ satisfies an exterior cone condition and is therefore
regular for the Laplace Dirichlet problem in the sense of Perron.

We can now apply the sub- and supersolution theorem of
Cl\'ement and Sweers \cite[Theorem, p.~191]{ClementSweers1987}.
To check its hypotheses on the nonlinearity, extend $f$ to
$\mathbb R$ by
\[
 \widetilde f(s)=(s_+)^{2p-1}-(s_+)^p,
 \qquad s_+=\max\{s,0\}.
\]
The function $(x,s)\mapsto\widetilde f(s)$ is continuous on
$\overline{\Omega}_k\times\mathbb R$, which verifies
assumption~(H1) of that theorem. The continuity of $g_k$ gives
(H2), and the Perron regularity of every boundary point gives
(H3). Therefore there exists $z_k\in C(\overline{\Omega}_k)$
such that
\[
\begin{cases}
 -\Delta z_k=f(z_k)\quad&\text{in }\Omega_k,\\
 z_k=g_k\quad&\text{on }\partial\Omega_k,\\
 0\le z_k\le h\le C_0.
 \end{cases}
\]
The equation holds in the sense of distributions; interior elliptic
regularity and the local Lipschitz continuity of $f$ give
$z_k\in C^2(\Omega_k)$.

We next pass to the limit while retaining the boundary trace.
Fix $R>0$. For $k>4R$, the function $w_k=z_k-u$ satisfies
\[
\begin{cases}
 -\Delta w_k=f(z_k)-f(u)
       &\quad\text{in }H\cap B_{2R}(te),\\
 w_k=0&\quad\text{on }\Pi\cap B_{2R}(te).
 \end{cases}
\]
Both $w_k$ and $\Delta w_k$ are bounded independently of
$k$. Local interior and (flat-)boundary $W^{2,q}$ estimates consequently give, for every $1<q<\infty$,
\[
 \|w_k\|_{W^{2,q}(H\cap B_R(te))}\le C(R,q),
\]
where $C(R,q)$ is independent of $k$; see
\cite[Chapter~9, Theorems~9.11 and~9.13]
{GilbargTrudinger}. Taking $q>n$ and using a diagonal subsequence, we obtain
$z_k\to z$ in $C^1$ on compact subsets of $\overline H$.
Thus 
\[
\begin{cases}
 -\Delta z=f(z)\quad&\text{in }H,\\
 z=u\quad&\text{on }\Pi,\\
 0\le z\le h
 \end{cases}
\]
in the sense of distributions in $H$. Interior regularity and local boundary Schauder regularity yield
$z\in C^{2,\alpha}_{\mathrm{loc}}(\overline H)$ for some
$\alpha\in(0,1)$ (see also
\cite[Lemma~6.18]{GilbargTrudinger}). Finally,
\[
 -\Delta z+C_0^{p-1}z
 =z^{2p-1}+\bigl(C_0^{p-1}-z^{p-1}\bigr)z\ge0.
\]
By the strong maximum principle, since  $z=u>0$ on $\Pi$, we have $z>0$ in $H$. This proves
\eqref{eq:lower-solution}. 

\medskip
\noindent\emph{Step 2.} Choose $\eta_R\in C_c^\infty(B_{7R/4})$, with $0\le\eta_R\le1$,
equal to one on $B_{5R/4}$, and $|\nabla\eta_R|\le C/R$.
Choose also $\chi_R\in C_c^\infty(A_R)$, with $0\le\chi_R\le1$,
equal to one on $\operatorname{supp}\nabla\eta_R$, and
$|\nabla\chi_R|\le C/R$. For $w=z-u$ in $H$, boundedness and \eqref{eq:lower-solution} give
\[
 |w|\le u,\qquad |\Delta u|+|\Delta w|\le Cu^p,
 \qquad w=0\ \text{on }\Pi.
\]
Testing the equation for $u$ with $\chi_R^2u$ in $\R^n$, and
the equation for $w$ with $\chi_R^2w$ in $H$ and using Young's inequality, we obtain the
Caccioppoli-type estimate
\begin{equation}\label{eq:reflection-caccioppoli}
 \int_{\R^n}\chi_R^2|\nabla u|^2\dd
 +\int_H\chi_R^2|\nabla w|^2\dd
 \le C\int_{A_R}\left(u^{p+1}+R^{-2}u^2\right)\dd.
\end{equation}
There is no boundary term on $\Pi$, since $w$ has zero trace there.
It follows that
$$
 \int_{H\cap\operatorname{supp}\nabla\eta_R}
 \left(|\nabla u|^2+|\nabla z|^2\right)\dd
 \le C\int_{A_R}\left(u^{p+1}+R^{-2}u^2\right)\dd.
$$
In particular, recalling the definition of $M(R)$ and the fact that $\operatorname{supp}\nabla\eta_R\subset A_R$, we obtain
\begin{equation}\label{eq:reflection-energy}
 \int_{H\cap\operatorname{supp}\nabla\eta_R}
 \left(|\nabla u|^2+|\nabla z|^2+u^{p+1}\right)\dd
 \le C\left(M(R)+R^{-2}\int_{A_R}u^2\dd\right).
\end{equation}

\medskip
\noindent\emph{Step 3:}
We consider the following vector field
$$
\mathcal{J}_e[u]:=\partial_e u \nabla u -\left(\frac12|\nabla u|^2-F(u)\right)e
$$
where $F$ is the anti-derivative of $f$ given by \eqref{def_f}, i.e. 
$$
F(u) =\frac{u^{2p}}{2p}-\frac{u^{p+1}}{p+1}\, .
$$
A direct computation (see also \cite[Proposition 1]{PucciSerrinVariational} with $a\equiv 0$, $h=e$ and $\mathcal{F}(x,u,p)=\frac{1}{2}|p|^2-F(u)$) gives 
$$
\mathrm{div}\left(\mathcal{J}_e[u] \right)=0\quad \text{ in } \mathbb{R}^n \, . 
$$
In the same way, 
$$
\mathrm{div}\left(\mathcal{J}_e[z] \right)=0\quad \text{ in } H \, ; 
$$
in particular, if we consider 
$$
J:= \mathcal J_e[z]-\mathcal J_e[u]\, ,
$$
we have 
$$
\mathrm{div}(J)=0 \quad \text{ in } H\, . 
$$
On $\Pi$, define
\[
 \delta_+:=\partial_e(u-z),\quad \text{ and } \quad
 \delta_-:=\partial_e(u^\sigma-z)\, .
\]
From \eqref{eq:lower-solution}, both quantities are nonnegative. In addition, since $u,z$ have the same trace
and tangential derivatives on $\Pi$, we have
\begin{align}\label{eq:normal-factorization}
 J\cdot e &=(\partial_ez)^2-\left(\frac12|\nabla z|^2-F(z)\right)-(\partial_eu)^2+\left(\frac12|\nabla u|^2-F(u)\right)\\
 &=\frac12\left((\partial_ez)^2-(\partial_eu)^2\right)\\
 &=\frac12\delta_+\delta_-\ge0\, , \qquad \text{on } \Pi
\end{align}
where we also used the fact that $\partial_eu^\sigma=-\partial_eu$
on $\Pi$. The outward normal to $H$ on $\Pi$ is $-e$, so integration by parts gives
\begin{align}\label{eq:relative-cutoff-identity}
  0\leq    \frac12\int_\Pi\delta_+\delta_-\eta_R\,d\sigma
&=\int_\Pi J\cdot e\,\eta_R \,d\sigma \nonumber\\
&=-\int_H \mathrm{div}(J)\eta_R \dd-\int_H J\cdot\nabla\eta_R\dd \nonumber\\
&= -\int_H J\cdot\nabla\eta_R\dd\, . 
\end{align}
Moreover, since $0<z\le u\le C_0$, we have
$$
|J|\le C(|\nabla u|^2+|\nabla z|^2+u^{p+1}) \quad \text{ in } H\, .
$$
Hence \eqref{eq:relative-cutoff-identity} with \eqref{eq:reflection-energy} yields
\begin{align}\label{eq:reflection-estimate}
 0&\le\frac12\int_{\Pi\cap B_R}\delta_+\delta_-\ds\\
 &\le \frac{C}{R}\int_{H\cap \mathrm{supp}\nabla\eta_R} |J|\dd \\
& \le C\left(\frac{M(R)}{R}+
                  \frac1{R^3}\int_{A_R}u^2\dd\right)\, .
\end{align}
Moreover, by H\"older's inequality,
\begin{equation}\label{eq:reflection-holder}
 \frac1{R^3}\int_{A_R}u^2\dd
 \le C R^{-\kappa}\left(\frac{M(R)}R\right)^{2/(p+1)},
 \qquad
 \kappa=\frac{4-(n-3)(p-1)}{p+1}>0,
\end{equation}
where the last inequality follows from \eqref{eq:range}. Thus the right-hand side of \eqref{eq:reflection-estimate}
tends to zero along $R_j$. By taking the limit $R_j\to\infty$ we obtain
$$
\delta_+\delta_-=0\quad\text{on } \Pi.
$$
The nonnegative functions $w_+=u-z$ and $w_-=u^\sigma-z$ satisfy 
$$
\begin{cases}
-\Delta w_\pm=c_\pm(x)w_\pm\geq -pC_0^{p-1}w_\pm&\quad\text{in }H\\
w_\pm=0&\quad\text{on }\Pi.
\end{cases}
$$
The strong maximum principle and Hopf's lemma imply that each $w_\pm$ either
vanishes identically or has strictly positive inward derivative at
every point of $\Pi$. Since $\delta_+\delta_-=0$, at least one
of these functions vanishes identically. Therefore
\begin{equation}\label{eq:reflection-order}
 u\le u^\sigma\ \text{in }H
 \quad\text{or}\quad u\ge u^\sigma\ \text{in }H.
\end{equation}
This holds for every hyperplane $\Pi$.

\medskip
\noindent\emph{Step 4.} We adapt the standard Hopf-lemma argument used 
in the classical moving planes method. For $x\in\R^n$ and a unit vector $e\in\mathbb{S}^{n-1}$, consider the hyperplane passing through $x$ with normal $e$, $\Pi_{x,e}$. By \eqref{eq:reflection-order}, either $u-u^\sigma$ or $u^\sigma-u$ is nonnegative in the corresponding half-space $H_{x,e}$ and satisfies a
linear equation with bounded coefficient. Therefore, Hopf's lemma implies that either $u- u^\sigma\equiv 0$ or 
$$
0<\pm\partial_e(u-u^\sigma)=\pm2\partial_e u\quad\text{on }\Pi_{x,e}.
$$ 
Thus,
\begin{equation}\label{eq:stationary-plane-symmetry}
 \text{if}\quad\partial_eu(x)=0
 \quad\Longrightarrow\quad
 \Pi_{x,e}\text{ is a symmetry plane of }u.
\end{equation}
Assume that $\partial_e u>0$ in $\R^n$ (the other case is similar). Since $u>0$, there exists $m>0$ such that
$$
u\geq m\quad\text{in }B_1(4ke)\quad\forall k\geq 0, k\in \mathbb{N}.
$$
By construction, for all $k$ the balls $B_1(4ke)$ are disjoint. Moreover, for all large enough $R$, there are at least $R/8$ balls of this type in $A_R$. Thus
$$
\frac{\omega_n m^{p+1}R}{8}\leq \int_{A_R}u^{p+1}\,dx = M(R),
$$
contradicting \eqref{eq:main-hyp}. Therefore, for every $e\in\mathbb{S}^{n-1}$, there exists $\hat x\in\R^n$ such that $\partial_e u(\hat x)=0$ and hence, by \eqref{eq:stationary-plane-symmetry}, $u$ is symmetric with respect to $\Pi_{\hat x,e}$. In particular $\partial_e u\equiv 0$ on $\Pi_{\hat x,e}$. Thus it follows that, for every $i=,\ldots, n$, there exists a hyperplane $\{x_i=const\}$ along which $\partial_i u\equiv 0$. Let $x_0$ be the intersection of those hyperplanes. Then $\nabla u(x_0)=0$, so by \eqref{eq:stationary-plane-symmetry} every hyperplane through $x_0$ is a symmetry plane. Hence
\begin{equation}\label{eq:radial-conclusion}
 u(x)=\varphi(|x-x_0|)
\end{equation}
for some $\varphi:[0,+\infty)\to\R$. The annular hypothesis \eqref{eq:main-hyp} implies $\inf\varphi(r)=0$: otherwise $u$ would have a positive lower bound, giving $M(R)\ge cR^n$, a contradiction.

We claim that $\varphi'(r)\ne0$ for every $r>0$.
Indeed, if $\varphi'(r_0)=0$, then, for any unit vector $e$,
\[
 \partial_e u(x_0+r_0e)=\varphi'(r_0)=0.
\]
By \eqref{eq:stationary-plane-symmetry}, the hyperplane $\Pi_{x_0+r_0e,e}$ is a symmetry plane of $u$. The parallel hyperplane through $x_0$, $\Pi_{x_0,e}$ is also a symmetry plane, by radial symmetry. Composing the two reflections makes $\varphi$ periodic with period $2r_0$. Since $\varphi>0$, there exists $m>0$ such that $\varphi\geq m$. Consequently,
\[
 M(R)=\int_{A_R}u^{p+1}\,dx
 \ge CR^n,
\]
contradicting \eqref{eq:main-hyp}.

Therefore $\varphi'$ has a constant strict sign
on $(0,\infty)$. This sign cannot be positive, since then
$\varphi(r)\ge\varphi(0)>0$, contradicting
$\inf\varphi(r)=0$. Hence $\varphi$ is strictly decreasing,
and $\varphi\to 0$ at infinity. We have obtained a positive radial ground state:
\[
 \varphi''+\frac{n-1}{r}\varphi'+f(\varphi)=0,\qquad
 \varphi'(0)=0,\qquad \lim_{r\to\infty}\varphi(r)=0.
\]
Uniqueness follows, for $n=2$, from 
\cite[Theorem~2]{PucciSerrinUniqueness} and for $n\ge3$ from 
\cite{SerrinTang}. A direct computation shows that the profile $\mathcal U_{x_0}$ in \eqref{eq:profile} is such
a solution and satisfies
\begin{equation}\label{eq:decay-mass}
 \int_{A_R}\mathcal U_{x_0}^{p+1}\dd
 =O\bigl(R^{n-2(p+1)/(p-1)}\bigr)=O(R^{-s}).
\end{equation}
\end{proof}

\section{Proof of Theorem~\ref{thm:sharp}}\label{sec:sharp}

\begin{proof}[Proof of Theorem~\ref{thm:sharp}]
We recall the quantities introduced in
Section~\ref{subsec:riem-gn-critical-rigidity}, written here
in the Euclidean setting: $v=u^{1-p}$,
\begin{equation}\label{eq:P-definition}
 P=\frac D{2p}\left(1+\frac{\Delta v}{p-1}\right)\, ,
\end{equation}
and 
$$
 T=\Hess v-\frac{(p-1)^2}{D}P\, \Id\, ,
$$
where $D$ is given by \eqref{eq:constants}. Since the Euclidean metric has vanishing Ricci curvature,
identity~\eqref{eq:riem-gn-P-divergence} becomes
\begin{equation}\label{eq:P-divergence}
 \diver(u^2\nabla P)
 =u^{p+1}\left[
   \frac D{(p-1)^2}|T|^2+(p-1)P(P-1)
 \right].
\end{equation}
Define the operator
\[
 \mathcal L\varphi
 :=u^{-(p+1)}\diver(u^2\nabla\varphi)
 =v\Delta\varphi
  -\frac{2}{p-1}\langle\nabla v,\nabla\varphi\rangle.
\]
Then \eqref{eq:P-divergence} reads
\begin{equation}\label{eq:P-equation}
 \mathcal LP
 =\frac D{(p-1)^2}|T|^2+(p-1)P(P-1).
\end{equation}
We also recall from \eqref{eq:critical-GN-P-expressions} that
\begin{align}
 P&=\frac D{2pv}
    +\frac D{2(p-1)^2v}|\nabla v|^2,
    \label{eq:P-pressure}\\
 P&=\frac D{2p}u^{p-1}
    +\frac D2\frac{|\nabla u|^2}{u^{p+1}}.
    \label{eq:P-solution}
\end{align}
In particular, $P>0$, and Corollary~\ref{cor:pressure-bounds}
implies that $P$ is bounded and
\[
 v(x)\le C_u(1+|x|^2),
 \qquad
 |\nabla v(x)|\le C_u(1+|x|).
\]
We first prove that $P\le1$. Set
\[
 \psi(x)=\log(1+|x|^2).
\]
Since
\[
 \nabla\psi=\frac{2x}{1+|x|^2},
 \qquad
 \Delta\psi=\frac{2n}{1+|x|^2}
            -\frac{4|x|^2}{(1+|x|^2)^2},
\]
the preceding growth estimates give
\begin{equation}\label{eq:sharp-barrier}
 |\mathcal L\psi|\le C_u
 \qquad\text{in }\R^n.
\end{equation}
Suppose by contradiction that $\sup_{\R^n}P>1$.
Choose $y\in\R^n$ and $\delta>0$ such that
$P(y)>1+2\delta$.
For every $\varepsilon>0$, the function
$P-\varepsilon\psi$ attains a global maximum at some
$x_\varepsilon\in\R^n$, because $P$ is bounded and
$\psi(x)\to\infty$ as $|x|\to\infty$.
Moreover, for $\varepsilon$ small enough,
\[
 P(x_\varepsilon)
 \ge P(x_\varepsilon)-\varepsilon\psi(x_\varepsilon)
 \ge P(y)-\varepsilon\psi(y)>1+\delta.
\]
At this maximum,
\[
 \nabla(P-\varepsilon\psi)(x_\varepsilon)=0,
 \qquad
 \Delta(P-\varepsilon\psi)(x_\varepsilon)\le0.
\]
Since $v>0$, it follows from \eqref{eq:sharp-barrier} that
\[
 \mathcal LP(x_\varepsilon)
 \le\varepsilon\mathcal L\psi(x_\varepsilon)
 \le C_u\varepsilon.
\]
On the other hand, \eqref{eq:P-equation} implies
\[
 \mathcal LP(x_\varepsilon)
 \ge(p-1)P(x_\varepsilon)(P(x_\varepsilon)-1)
 \ge(p-1)(1+\delta)\delta.
\]
Letting $\varepsilon\to 0$ gives a contradiction.
Thus $P\le1$, which is precisely \eqref{eq:sharp-u}
by \eqref{eq:P-solution}.

Suppose now that $P(x_*)=1$ at some point $x_*\in\mathbb{R}^n$, and set
$f=1-P\ge0$. Equation~\eqref{eq:P-equation} gives
\[
 \mathcal Lf-(p-1)Pf
 =-\frac D{(p-1)^2}|T|^2\le0.
\]
Since $P>0$ and $f(x_*)=0$, the strong maximum principle implies that $f\equiv0$, i.e. $P\equiv1$. \eqref{eq:P-equation} gives $T\equiv0$, i.e. 
\[
 \Hess v=\frac{(p-1)^2}{D}\Id.
\]
Integrating and completing the square, we obtain
\[
 v(x)=A+\frac{(p-1)^2}{2D}|x-x_0|^2
 \qquad\text{for some } x_0\in\R^n,\quad A=v(x_0)>0.
\]
Since $P\equiv 1$ and $\nabla v(x_0)=0$, \eqref{eq:P-pressure} gives
\[
A=\frac{D}{2p}\, ,
\]
and 
\[
 u(x)=
 \left[\frac D{2p}+\frac{(p-1)^2}{2D}|x-x_0|^2\right]^{-1/(p-1)}
 =\mathcal U_{x_0}(x),
\]
as claimed.
\end{proof}

\section{Examples and comments}

\subsection{Explicit critical points on nonflat Riemannian manifolds}
\label{sec:nonflat-GN-critical-examples}

The Ricci curvature assumption in
Theorem~\ref{thm:riem-gn-critical-rigidity} is essential. The following
construction gives positive critical points of the intrinsic
Gagliardo--Nirenberg quotient \eqref{eq:critical-GN-quotient} in every dimension $n\ge2$ and for
every exponent in \eqref{eq:range}. The metrics have two ends and
polynomial volume growth of order strictly greater than $n$.

\begin{proposition}\label{ex:explicit-GN-all-dimensions}
Let $n\ge2$, let $1<p<\frac{n}{n-2}$ and define 
\begin{equation}\label{eq:explicit-GN-volume-parameter}
 \delta=\frac{\sqrt{(p+1)^2+\frac{8D(p+1)}{n}}-(p+1)}{2},
 \qquad
 \nu=\frac{2p-\delta}{p-1}.
\end{equation}
where $D$ is given by \eqref{eq:constants}. On
$M=\R\times\mathbb S^{n-1}$, consider the metric
\begin{equation}\label{eq:explicit-GN-metric-profile}
 g=dt^2+(1+t^2)^{\frac{\nu-1}{n-1}}g_{\mathbb S^{n-1}}\,,
\end{equation}
and the function
$$
u(t,\omega)=(1+t^2)^{-\frac{1}{p-1}}\, , \quad \text{ for } (t,\omega)\in M\, .
$$
Then the manifold $(M,g)$ is smooth, complete and nonflat, and
$u\in\mathcal D^p(M)$ is a positive critical point of
$\mathscr G_{n,p}$. Moreover, for every fixed $o\in M$,
\begin{equation}\label{eq:explicit-GN-ball-volume}
 \operatorname{Vol}_g(B_R(o))
 \sim\frac{2|\mathbb S^{n-1}|}{\nu}R^\nu\, ,
 \quad\text{as }R\to\infty.
\end{equation}
In particular, since $\nu>n$ we have $\mathrm{AVR}_g=+\infty$.
\end{proposition}

\begin{proof}
The number $\delta$ is the positive root of
\begin{equation}\label{eq:explicit-GN-parameter-identity}
 \delta(p+1+\delta)=\frac{2D(p+1)}{n}.
\end{equation}
Since $x\mapsto x(p+1+x)$ is strictly increasing for $x\ge0$ and
\[
 D(p+1+D)>\frac{2D(p+1)}{n},
\]
we have $0<\delta<D$. Thus, from \eqref{eq:constants} we have  
$$
\nu=\frac{2p-\delta}{p-1}>\frac{2p-D}{p-1}=n\, . 
$$
The warping function is
$$
\psi(t):=(1+t^2)^{\frac{\nu-1}{2(n-1)}}\ge1\, , 
$$
so $g$ is smooth and, since $g\geq dt^2+g_{\mathbb{S}^{n-1}}$, it is complete. Direct calculation (see e.g. \cite{Petersen}) gives
\begin{equation}\label{eq:explicit-GN-Ricci}
 \operatorname{Ric}_g(\partial_t,\partial_t)
 =-(n-1)\frac{\psi''}{\psi}
 =-\frac{\nu-1}{(1+t^2)^2}
    \left(1+\frac{\nu-n}{n-1}t^2\right)<0\, . 
\end{equation}
The volume element and the Laplacian of a function $f=f(t)$ are
\begin{equation}\label{eq:explicit-GN-laplacian}
 \begin{gathered}
 d\mu_g=(1+t^2)^{\frac{\nu-1}{2}}\,dt\,d\omega,\\
 \Delta_gf
 =(1+t^2)^{-\frac{\nu-1}{2}}\frac{d}{dt}
       \left((1+t^2)^{\frac{\nu-1}{2}}f'\right)
 =f''+\frac{(\nu-1)t}{1+t^2}f'.
 \end{gathered}
\end{equation}
For $u=(1+t^2)^{-1/(p-1)}$ we have
\begin{align}
     u'&=-\frac{2t}{p-1}(1+t^2)^{-\frac{p}{p-1}} \, ,\\
 \Delta_gu
 &=\frac{2(2p-(p-1)\nu)}{(p-1)^2}u^p
  -\frac{2(2p-(p-1)\nu+p-1)}{(p-1)^2}u^{2p-1}.
\end{align}
Since $2p-(p-1)\nu=\delta$, we have 
\begin{equation}\label{eq:explicit-GN-equation}
 -\Delta_g u=\mu u^{2p-1}-\lambda u^{p} \quad \text{ in } M\, , 
\end{equation}
where 
$$
\lambda=\frac{2\delta}{(p-1)^2}\quad \text{ and }\quad 
 \mu=\frac{2(\delta+p-1)}{(p-1)^2}\, . 
$$
We now prove that $u\in\mathcal{D}^p(M)$ and is a critical point of $\mathscr{G}_{n,p}$. We set (as in the proof of Theorem \ref{thm:riem-gn-critical-rigidity})
\[
 I_1=\int_M|\nabla u|^2\,d\mu_g,\qquad
 I_2=\int_Mu^{p+1}\,d\mu_g,\qquad
 I_3=\int_Mu^{2p}\,d\mu_g,
\]
and a direct computation shows that
\begin{align*}
 I_1&=\frac{4|\mathbb S^{n-1}|}{(p-1)^2}
                  \int_{\R}t^2(1+t^2)^{-k-1}\,dt,\\
 I_2&=|\mathbb S^{n-1}|\int_{\R}(1+t^2)^{-k}\,dt,\\
 I_3&=|\mathbb S^{n-1}|\int_{\R}(1+t^2)^{-k-1}\,dt\, , 
\end{align*}
where 
\[
 k=\frac{p+1}{p-1}-\frac{\nu-1}{2}
   =\frac{p+1+\delta}{2(p-1)}>\frac12\, .
\]
Observe that all three integrals are finite, i.e. $u\in \mathcal D^p(M)$, and one has
\begin{equation}\label{eq:explicit-GN-integral-ratios}
 \frac{I_1}{I_2}=\frac{2}{(p-1)^2k},
 \qquad
 \frac{I_3}{I_2}=\frac{2k-1}{2k}, \qquad \frac{I_1}{I_3}=\frac{4}{(p-1)^2(2k-1)}\, . 
\end{equation}
From \eqref{eq:explicit-GN-parameter-identity} and the definition of $k$ we have 
\[
 k\delta=\frac{D(p+1)}{n(p-1)}=\frac{1-\tau}{\tau}\, , 
\]
and 
$$
(\delta+p-1)(2k-1)=2(1+k\delta)=\frac{2}{\tau}\, , 
$$
where $\tau$ is given by \eqref{eq:critical-GN-parameters}. Consequently,
\begin{equation}\label{l_and_mu_new}
     \frac{(1-\tau)I_1}{\tau I_2}
 =\frac{2\delta}{(p-1)^2}=\lambda,
 \quad
 \frac{I_1}{\tau I_3}
 =\frac{2(\delta+p-1)}{(p-1)^2}=\mu.
 \end{equation}
Thus, from \eqref{l_and_mu} we get that \eqref{eq:explicit-GN-equation} is the Euler--Lagrange equation of $\mathscr G_{n,p}$. 

Fix $o=(0,\omega_0)\in M$, then 
$$
 \operatorname{Vol}_g(B_R(o))\sim 2|\mathbb S^{n-1}|\int_0^R(1+t^2)^{\frac{\nu-1}{2}}\,dt  \sim\frac{2|\mathbb S^{n-1}|}{\nu}R^\nu \quad\text{as } R\to+\infty,
$$
which proves \eqref{eq:explicit-GN-ball-volume} at this center.
Changing the center changes the radius by a fixed constant, so the same asymptotic holds for every fixed $o\in M$. Since $\nu>n$, from the definition of asymptotic volume ratio it follows immediately that $\mathrm{AVR}_g=+\infty$.
\end{proof}

\subsection{Consequences and sharp examples}\label{sec:consequences}

We start by showing Corollary~\ref{cor:finite-mass}. 

\begin{proof}[Proof of Corollary~\ref{cor:finite-mass}]
The equivalence between \textup{(i)} and \textup{(v)} follows from
Theorem~\ref{thm:main}. For every profile \eqref{eq:profile}, a direct computation shows that
$$
 \mathcal U_{x_0}(x)=O\bigl(|x|^{-\frac{2}{p-1}}\bigr),
 \quad \text{ and } \quad 
 |\nabla\mathcal U_{x_0}(x)|
 =O\bigl(|x|^{-\frac{p+1}{p-1}}\bigr)\, ,
$$
as $|x|\to\infty$. Consequently,
$$
 |\nabla\mathcal U_{x_0}(x)|^2+\mathcal U_{x_0}^{p+1}(x)=O\bigl(|x|^{-\frac{2(p+1)}{p-1}}\bigr)
 =O\bigl(|x|^{-n-s}\bigr)\, ,
 $$
 and 
 $$
 \mathcal U_{x_0}^{2p}(x)=O\bigl(|x|^{-\frac{4p}{p-1}}\bigr)
 =O\bigl(|x|^{-n-s-2}\bigr)\, , 
 $$
 where we used the definition of $s$ in \eqref{eq:s}.
Since $s>0$ and $\mathcal U_{x_0}$ is smooth, these estimates show that
\textup{(v)} implies both \textup{(iii)} and \textup{(iv)}.
Moreover, \textup{(iii)} implies \textup{(ii)} by definition.
If \textup{(ii)} holds, then
\[
 0\le M(R)
 \le \int_{\R^n\setminus B_R}u^{p+1}\dd
 \longrightarrow0,
\]
and hence \textup{(i)} follows. It remains to prove that \textup{(iv)} implies \textup{(i)}.
Assume that $u(x)\to0$ as $|x|\to\infty$. Choose $R_0>0$ such that
$$
u^{p-1}\le\frac{1}{2} \quad \text{in }  \R^n\setminus B_{R_0} \, ,
$$ 
then 
\[
 \Delta u=u^p(1-u^{p-1})\ge\frac12u^p \quad \text{in }  \R^n\setminus B_{R_0} \, . 
\]
For $|x|\ge2R_0$, the ball $B_{|x|/4}(x)$ is contained in
$\R^n\setminus B_{R_0}$. The local Keller--Osserman estimate recalled in Section \ref{sec:estimates} on
this ball yields
\[
 u(x)\le c|x|^{-\frac{2}{p-1}}\, . 
\]
Therefore, for all sufficiently large $R$,
\[
 M(R)
 \le C R^{\,n-\frac{2(p+1)}{p-1}}
 =C R^{-s}\, ,
\]
where $s$ is given by \eqref{eq:s}. Since $s>0$, this proves \textup{(i)} and completes the equivalence
of the five conditions.

Finally, let $u$ be a nonnegative classical solution. We may write
\[
 \Delta u=a(x)u,
 \qquad
 a(x)=u(x)^{p-1}-u(x)^{2p-2},
\]
where $a$ is locally bounded. The strong maximum principle,
applied locally, and the connectedness of $\R^n$ imply that either
$u\equiv0$ or $u>0$ everywhere. In the latter case, any of
\textup{(i)}--\textup{(iv)} gives \textup{(v)} by the equivalence
just proved.
\end{proof}

The next proposition shows that the threshold in Theorem \ref{thm:main} is optimal.
\begin{proposition}\label{prop:cylinders}
Let $1\le m<n$, write $x=(y,z)\in\R^m\times\R^{n-m}$, and set
\begin{equation}\label{eq:cylindrical-profile}
 \mathcal U_m(y)=\left(\frac{D_m}{2p}
              +\frac{(p-1)^2}{2D_m}|y|^2\right)^{-\frac{1}{p-1}}\, , 
\end{equation}
where $D_m=2p-m(p-1)$. Then $u(y,z)=\mathcal U_m(y)$ is a positive solution of
\eqref{eq:main} and 
\begin{equation}\label{P_leq1}
P=\frac{D}{2p}\mathcal U_m^{p-1}
 +\frac{D}{2}\frac{|\nabla \mathcal U_m|^2}{\mathcal U_m^{p+1}}\equiv\frac{D}{D_m}(<1)
 \quad\text{in }\R^n\, .     
\end{equation}
Moreover, 
\begin{equation}\label{eq:cylindrical-mass}
\int_{A_R}\mathcal U_m^{p+1}\, dx\sim\omega_{n-m}(2^{n-m}-1)
          \left(\int_{\R^m}\mathcal U_m^{p+1}\,dy\right)R^{n-m} \, ,
 \quad\text{as }R\to\infty\, . 
\end{equation}
In particular, for $m=n-1$ the annular mass $M(R)$ is linear.
\end{proposition}
\begin{proof}
Since $D_m=D+(n-m)(p-1)>0$, the profile is well defined. A direct computation shows that $\mathcal U_m$ solves \eqref{eq:main} for every $1\leq m<n$. Moreover, \eqref{P_leq1} easily follows. Since, from \eqref{eq:s}, $2(p+1)/(p-1)=n+s>m$, we have $\mathcal U_m^{p+1}\in L^1(\R^m)$.
By Fubini,
\[
\int_{B_R}\mathcal U_m^{p+1}\dd
 =\omega_{n-m}\int_{\R^m}\mathcal U_m(y)^{p+1}
       \left(R^2-|y|^2\right)_+^{\frac{n-m}{2}}\,dy.
\]
Dominated convergence gives \eqref{eq:cylindrical-mass}. 
\end{proof}

\begin{remark}\label{ex:periodic}
The solution $u\equiv1$ has $P=D/(2p)<1$ and $M(R)$ proportional
to $R^n$. Positive nonconstant periodic solutions can also be extended
from one variable. The equation
\[
 w''=w^p-w^{2p-1}
\]
has conserved energy
\[
 \frac12(w')^2+\frac{w^{2p}}{2p}-\frac{w^{p+1}}{p+1}.
\]
Its potential has a strict local minimum at $w=1$, with second
derivative $p-1>0$. Nearby energy levels give positive periodic
oscillations bounded away from zero. The corresponding solutions
$u(x)=w(x_1)$ have annular mass comparable to $R^n$.
\end{remark}


\subsection{Solutions of the normalized equation on Riemannian manifolds}
\label{sec:riemannian}

The following examples show that positive sectional curvature and an
asymptotic volume ratio arbitrarily close to one do not force rigidity
for positive solutions of the normalized equation, even under decay
of the annular mass.
\begin{proposition}
\label{prop:positive-avr-counterexamples}
Let $n\ge2$, let $1<p<\frac{n}{n-2}$ and let $c\in(0,1)$. The metric
\begin{equation}\label{eq:positive-avr-metric}
 g_c=dr^2+\Psi_c(r)^2g_{\mathbb S^{n-1}},
 \qquad \Psi_c(r)=cr+(1-c)\tanh r,
\end{equation}
on $\R^n$ is smooth and complete, has strictly positive sectional
curvature, and satisfies
\begin{equation}\label{eq:positive-avr-value}
 \operatorname{AVR}_{g_c}=c^{n-1}.
\end{equation}
Moreover, the equation 
$$
-\Delta_{g_c}u=u^{2p-1}-u^p \quad \text{ in } \mathbb{R}^n
$$ 
admits a smooth positive radial solution $u_c$ such that
\begin{equation}\label{eq:uc-decay}
 u_c'(r)<0\quad \text{ for } r>0 \, , \quad \text{and} \quad \lim_{r\to\infty}u_c(r)=0\, .
\end{equation}
Finally, for balls centered at the pole,
\begin{equation}\label{eq:uc-annular}
 \int_{B_{2R}\setminus B_R}u_c^{p+1}\,d\mu_{g_c}
 =O(R^{-s})\, .
\end{equation}
where $s$ is given by \eqref{eq:s}.
\end{proposition}

\begin{proof}
The function $\Psi_c$ is smooth and odd, with
$$
\Psi_c(r)=r-(1-c)\frac{r^3}{3}+O(r^5)\, , 
$$
as $r\to 0$, so the metric extends smoothly
across the pole and is complete. Moreover, since 
$$
0<\Psi_c'<1\quad \text{ and } \quad \Psi_c''<0 \, \text{ on } (0,\infty)\, ,
$$ 
we have (see e.g. \cite{Petersen})
\[
 \mathrm{Sec}_{\mathrm{rad}}=-\frac{\Psi_c''}{\Psi_c}>0,
 \qquad
 \mathrm{Sec}_{\mathrm{tan}}=\frac{1-(\Psi_c')^2}{\Psi_c^2}>0,
\]
where the second formula is relevant when $n\ge3$. At the pole,
the sectional curvature is $2(1-c)$. Finally, 
$\Psi_c(r)/r\to c$, as $r\to+\infty$, gives \eqref{eq:positive-avr-value}.

Existence follows from the mountain pass theorem of
Ambrosetti--Rabinowitz \cite[Theorem~2.1]{AmbrosettiRabinowitz}.
We verify its hypotheses in the radial energy space. Write
$d\mu=d\mu_{g_c}$ and set
\[
 X_c=\bigl\{v\in L^{p+1}(d\mu):
       v\text{ is radial},\ \nabla v\in L^2(d\mu)\bigr\},
\]
with $\|v\|_{X_c}=\|\nabla v\|_{L^2(d\mu)}+\|v\|_{L^{p+1}(d\mu)}$. For radial functions, $|\nabla v|=|v'|$, and
\[
 c^{n-1}r^{n-1}\le q(r):=\Psi_c(r)^{n-1}\le r^{n-1}.
\]
Thus $\|\cdot\|_{X_c}$ is equivalent to the Euclidean one. Since
\[
 p+1<2p<2^*,
\]
the radial compactness consequence of Lions's lemma
\cite[Lemma~I.1]{LionsCompactnessII} gives
\[
 X_c\hookrightarrow\hookrightarrow L^{2p}(d\mu).
\]
Consider the $C^1$ functional
\[
 \mathcal I_c(v)
 =\int_{\R^n}
   \left(\frac12|\nabla v|^2
         +\frac{|v|^{p+1}}{p+1}
         -\frac{(v_+)^{2p}}{2p}\right)\,d\mu.
\]
The continuous embedding into $L^{2p}(d\mu)$ and
$2<p+1<2p$ imply that $\mathcal I_c$ is positive on a sufficiently
small sphere about zero, whereas
$\mathcal I_c(t\varphi)\to-\infty$ as $t\to\infty$ for every
nonzero nonnegative radial $\varphi\in C_c^\infty(\R^n)$.
Moreover,
\[
 2p\,\mathcal I_c(v)-\mathcal I_c'(v)[v]
 =(p-1)\int_{\R^n}
       \left(|\nabla v|^2+\frac{|v|^{p+1}}{p+1}\right)\,d\mu
\]
shows that Palais--Smale sequences are bounded in $X_c$. By the compact embedding, every bounded Palais--Smale sequence
$\{v_j\}$ admits a subsequence such that $v_j\rightharpoonup v$
in $X_c$ and $v_j\to v$ in $L^{2p}(d\mu)$.
Testing $\mathcal I_c'(v_j)-\mathcal I_c'(v)$ against $v_j-v$
and using
\[
 (|a|^{p-1}a-|b|^{p-1}b)(a-b)
 \ge 2^{1-p}|a-b|^{p+1},
\]
we obtain
\[
 \|\nabla(v_j-v)\|_{L^2(d\mu)}^2
 +2^{1-p}\|v_j-v\|_{L^{p+1}(d\mu)}^{p+1}=o(1).
\]
Thus $v_j\to v$ in $X_c$, proving the Palais--Smale condition. The mountain pass theorem therefore provides a nonzero radial
critical point $u_c\in X_c$.

Since $g_c$ is rotationally symmetric, the functional $\mathcal I_c$
is invariant under the natural action of the compact group $O(n)$.
Palais' principle of symmetric criticality \cite{Palais1979}
therefore implies that the radial critical point $u_c$ is a critical
point with respect to arbitrary variations, hence a weak solution of
\[
 -\Delta_{g_c}u_c+|u_c|^{p-1}u_c=(u_c)_+^{2p-1}.
\]
Testing with its negative part gives $u_c\ge0$. Local elliptic
regularity and the strong maximum principle yield
$u_c\in C^\infty(\R^n)$ and $u_c>0$.
By H\"older's inequality, we obtain
\[
 |u_c(r)|^{(p+3)/2}
 \le C r^{1-n}\|\nabla u_c\|_{L^2(d\mu)}
                  \|u_c\|_{L^{p+1}(d\mu)}^{(p+1)/2},
\]
which implies $u_c(r)\to0$ as $r\to\infty$.

To establish monotonicity and the required decay, set
\[
 f(t)=t^{2p-1}-t^p,
 \qquad
 F(t)=\frac{t^{2p}}{2p}-\frac{t^{p+1}}{p+1}.
\]
The radial equation reads
\begin{equation}\label{eq:shooting-ode}
 (qu_c')'+qf(u_c)=0,\qquad u_c'(0)=0,
\end{equation}
and its energy satisfies
\[
 E(r):=\frac12u_c'(r)^2+F(u_c(r)),
 \qquad
 E'(r)=-\frac{q'(r)}{q(r)}u_c'(r)^2\le0.
\]
Since $u_c(r)\to0$ as $r\to\infty$ and $\int_0^\infty q\,u_c'^2\,dr<\infty$,
we have $E(r)\to0$, hence $E\ge0$.
In particular $F(u_c(0))\geq 0$ and hence $u_c(0)>1$, so $u_c'<0$ near the pole.
A first critical point at a positive radius would satisfy
$u_c''\ge0$, hence $u_c\le1$ by \eqref{eq:shooting-ode};
at that point $E=F(u_c)<0$, a contradiction.
This proves \eqref{eq:uc-decay}.

For sufficiently small $t>0$,
$F(t)\le-t^{p+1}/(2(p+1))$. Consequently, $E\ge0$ and $u_c'<0$
give, for all sufficiently large $r$,
\begin{equation}\label{eq:uc-KO}
 -u_c'(r)\ge\frac{u_c(r)^{(p+1)/2}}{\sqrt{p+1}},
 \qquad
 u_c(r)\le Cr^{-2/(p-1)},
\end{equation}
where the second estimate follows by integration.
Using $\Psi_c(r)\le r$, we conclude that
\[
 \int_{B_{2R}\setminus B_R}u_c^{p+1}\,d\mu_{g_c}
 \le C\int_R^{2R}r^{n-1-2(p+1)/(p-1)}\,dr
 \le CR^{-s},
\]
as claimed.
\end{proof}

\

\subsection*{Acknowledgements} The first and third authors are members of the GNSAGA, Gruppo Nazionale per le Strutture Algebriche, Geometriche e le loro Applicazioni of INdAM. The second author is a member of GNAMPA, Gruppo Nazionale per l'Analisi Matematica, la Probabilit\`a e le loro Applicazioni of INdAM.

\

\subsection*{AI disclosure}The authors acknowledge the use of AI tools during the preparation of this paper. All mathematical arguments and proofs in the final manuscript were checked and written by the authors.


\

\subsection*{Conflict of interest}
The authors declare that they have no conflict of interest.

\end{document}